\documentclass[11pt]{amsart}
\usepackage[margin=1.15in]{geometry}
\usepackage{amsmath,amssymb,amsthm}
\usepackage{booktabs}
\usepackage{algorithm}
\usepackage{algpseudocode}
\usepackage{xcolor}
\usepackage[colorlinks=true,linkcolor=blue,citecolor=blue,urlcolor=blue]{hyperref}
\usepackage{microtype}
\usepackage{graphicx}

\newtheorem{theorem}{Theorem}[section]
\newtheorem{lemma}[theorem]{Lemma}

\theoremstyle{definition}

\newtheorem{remark}[theorem]{Remark}
\newtheorem{observation}[theorem]{Observation}

\newcommand{\R}{\mathbb{R}}
\newcommand{\Q}{\mathbb{Q}}
\newcommand{\bd}{\partial}
\newcommand{\com}{\mathbf{c}}
\newcommand{\shu}{(S,H,U)}
\newcommand{\repourl}{https://github.com/Tancredi-Schettini-Gherardini/certified_gomboc}

\title[A certified mono-monostatic polyhedron]{An explicit mono-monostatic polyhedron}

\author{Tancredi Schettini Gherardini}
\address{Mathematisches Institut, Universit\"at Bonn}
\email{tsg@math.uni-bonn.de}

\thanks{The construction, software, and certification described in this paper were produced in
a human--AI collaboration with Claude Opus 4.8 and Fable 5; see Section~\ref{sec:workflow} for a precise account of the
division of labour.}

\thanks{\emph{Code and data.} The half-space data of both certified polytopes, the
certificates and the verifier are available at the repo \href{\repourl}{\nolinkurl{\repourl}}; see Section~\ref{sec:data}.}

\begin{document}

\begin{abstract}
A convex body is \emph{mono-monostatic} if, resting under gravity on a horizontal plane,
it has exactly one stable and one unstable equilibrium position. Smooth
mono-monostatic homogeneous bodies exist (the G\"omb\"oc of Domokos and V\'arkonyi), and L\'angi proved
that (homogeneous) mono-monostatic \emph{polyhedra} exist; although no explicit example
appears to have been published, to the author's knowledge. We construct explicitly two such mono-monostatic polytopes, 
the smaller one having $56{,}946$ faces; importantly, we \emph{certify} them: the
polytope is presented as an intersection of half-spaces with rational data, and a
certifying verification establishes, using exact rational arithmetic for every decisive
comparison, that the body has equilibrium signature $(S,H,U)=(1,0,1)$ with respect to its
own exact centroid, with explicit nondegeneracy margins. We describe the geometric
obstructions that make naive discretisations of smooth mono-monostatic bodies fail, the
adaptive construction that overcomes them, and the certification strategy. While the result
itself is not strikingly novel, we emphasise the non-standard (but increasingly more common) methodology: the entire
programme, i.e.~experiments, constructions and the verifier itself, was implemented
by AI agents under human mathematical direction. We argue that
\emph{exact certification of numerically discovered objects} is the natural contract
between such AI-assisted workflows and mathematical standards of rigour.
\end{abstract}

\maketitle

\section{Introduction}\label{sec:intro}

A convex body $K\subset\R^3$ with centre of mass $\com$ rests in equilibrium on a
horizontal plane exactly at the points $x\in\bd K$ where the outward support direction is
parallel to $x-\com$. Generically these equilibria are nondegenerate critical points of
the distance function $\delta(x)=|x-\com|$ restricted to $\bd K$, and they come in three
kinds: local minima (stable rest positions), saddles, and local maxima. Writing $S$, $H$,
$U$ for their numbers, the Poincar\'e--Hopf theorem forces
\begin{equation}\label{eq:PH}
  S-H+U=2 .
\end{equation}
The minimal signature $(S,H,U)=(1,0,1)$ (one stable and one unstable equilibrium, no
saddle) is famously realised by the \emph{G\"omb\"oc} of Domokos and V\'arkonyi
\cite{DV2006,VD2006}, answering a question of V.\,I.~Arnold. Such bodies are called
\emph{mono-monostatic}.

For polyhedra the situation is more delicate. A tetrahedron has always at least two
stable faces \cite{Dawson1985}; Conway and Guy constructed a \emph{monostable} (one
stable face) polyhedron with $19$ faces, and the current record, due to Reshetov, has
$14$ faces \cite{Reshetov2014}. Monostability constrains only $S=1$; Reshetov's body,
as we certify below in exact arithmetic, has signature $(1,1,2)$. Whether \emph{mono-monostatic} polyhedra
exist at all was settled affirmatively by L\'angi \cite{Langi2022}, who proved that a
centred convex body can be approximated uniformly by mono-monostatic polyhedra if and
only if it is itself of type $(1,1)$ in the appropriate sense. L\'angi's argument is
constructive in principle but was, to our knowledge, never instantiated: no explicit
mono-monostatic polyhedron appears in the literature. The question of the minimal number
of faces, known informally as the \emph{G\"omb\"ocedron} problem and for which G.~Domokos has
offered a prize scaling inversely with the face count, is wide open, with the known
constraints leaving an enormous range.\footnote{The prize is concerned with the minimum number of faces, edges and vertices, not a
construction, and therefore our contribution at this point is simply an
upper bound. The construction presented here establishes that whoever
eventually determines the minimum will receive at least \$2.93. Meanwhile we offer,
for each further improvement of the record, $3\times\log_{10}
(\frac{56{,}946}{\#\mathrm{faces}})$ cents of our own; this might increase if we get tenured.}

We should explicitly emphasise that everything discussed above and in the remainder of this paper concerns \emph{homogeneous} bodies.
For inhomogeneous polyhedra, the situation is very different and largely understood: a monostable inhomogeneous tetrahedron, conjectured by Conway,
was recently contructed \cite{ADD2025}; inhomogeneous mono-monostatic tetrahedra do not exist, but every other allowed face vector is realised, so five
faces and five vertices are sufficient \cite{ADDR2023}; moreover, concering unit point masses located at the vertices, a mono-monostatic $0$-skeleton with 
$21$ faces and $21$ vertices is known \cite{DK2023}, and the minimal
mono-unstable polyhedron has been identified, with $8$ faces and $11$ vertices
\cite{BDKR2022,PRDB2023}. We note that many of these results rely on exact certificates, in the same spirit as the certification presented here.
Returning to homogeneous bodies, which is the subject of our investigation, no exact minimum is known.

This paper reports two things.

\smallskip
\emph{(i) An explicit construction.} We build mono-monostatic polytopes as adaptive
tangent covers of an explicit smooth mono-monostatic body, correcting two geometric
obstructions (Section~\ref{sec:construction}) that defeat naive discretisation: a
near-critical \emph{ridge} that forces locally fine facets, and a \emph{sag dipole} that
displaces the centroid of any circumscribed polytope beyond the tolerance that the ridge
permits. The smaller of our two bodies has
\[
  F = 56{,}946 \text{ faces}, \qquad E=170{,}832, \qquad V=113{,}888 .
\]

\smallskip
\emph{(ii) A certificate.} The construction is numerical, and we regard no numerical
claim about $50$ near-degenerate margins per thousand faces as trustworthy by itself.
We therefore treat the constructed half-space data list of $56{,}946$ pairs
$(n_i,h_i)$ of double-precision numbers, i.e.\ \emph{exact rationals}, as the
definition of a rational polytope, and verify the signature $(1,0,1)$ by a certifying
algorithm (Section~\ref{sec:certification}) in which every decisive comparison is either
carried out in exact rational arithmetic or protected by a rigorous floating-point error
bound with an exact-arithmetic fallback. The outcome is:

\begin{theorem}\label{thm:main}
Let $P=\bigcap_{i=1}^{56946}\{x\in\R^3 : \langle n_i,x\rangle\le h_i\}$, where
$(n_i,h_i)$ are the rational values published in the data set accompanying this paper
(SHA-256 \texttt{3b9618\ldots}). Then $P$ is a convex polytope with $113{,}888$ vertices,
all simple, $170{,}832$ edges and $56{,}946$ facets, of volume
$4.2185974\ldots$, and with respect to its exact uniform-density centroid it has exactly
one stable equilibrium (a facet, with foot-of-perpendicular margin
$\ge 4.6\cdot10^{-3}$), exactly one unstable equilibrium (a vertex, with normal-cone
margins $\ge 1.69\cdot10^{-8}$), and no saddle equilibrium; every equilibrium is
nondegenerate. In particular $(S,H,U)=(1,0,1)$.
\end{theorem}

The analogous statement, with the margins of Table~\ref{tab:bodies}, holds for our
less-refined construction with $F=272{,}971$. The certificate is reproducible in
minutes on a laptop.

We stress the intended reading. The theorem is a finite statement about a finite
rational object, established by a finite computation; the mathematical content of the
verification reduces to elementary convex geometry (Section~\ref{sec:certification})
plus the correctness of an approximately $700$-line verifier, of exact rational
arithmetic, and of IEEE-754 floating point used only within proven error budgets. This
is the standard of \emph{certifying algorithms} \cite{MMNS2011}.

Finally, Section~\ref{sec:workflow} documents how the result was obtained: a
human-AI workflow in which the author posed the problem, set standards and redirected
strategy, while a general-purpose AI agent (Claude Opus 4.8 and Fable 5 were both used at different stages of the project) 
designed and ran the experiments, discovered
and repaired the obstructions. We believe this experience is broadly relevant: as
AI systems become capable of extended autonomous mathematical experimentation, exact
certification is the mechanism that converts their plausible numerical artefacts into
mathematics.

\subsubsection*{Acknowledgements}
The author thanks Geordie Williamson for introducing them to the G\"omb\"oc,
and the fascinating story behind it, and Gabor Doomokos for his great insights and
precious comments on the manuscript. The author acknowledges the support of the 2024
Max Planck-Humboldt Research Award, bestowed on Geordie Williamson by
the Max Planck Society and the Alexander von Humboldt Foundation and
hosted by Catharina Stroppel at the University of Bonn.

\begin{figure}
\includegraphics[width=1\textwidth]{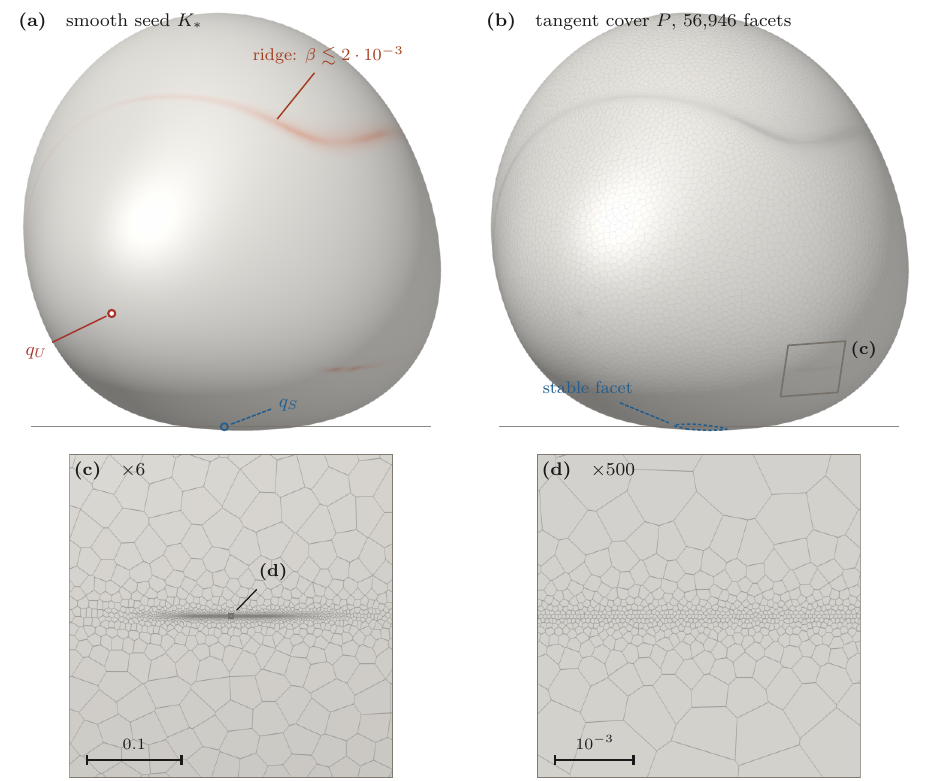}  
\caption{The $56{,}946$ facets \emph{G\"omb\"ocedron}. \textbf{(a)} The smooth G\"omb\"oc seed $K_\ast$, resting on its stable equilibrium $q_S$; 
$q_U$ is the unstable equilibrium, and
the near-critical ridge is highlighted in red, where the tilt $\beta$ drops below $2\cdot10^{-3}$.
\textbf{(b)} The polytope $P$ of Theorem~\ref{thm:main}, the equal-sag tangent
cover of $K_\ast$ with $56{,}946$ facets, in the same position and at the same scale,
with hairline facet edges; the dashed contour is its single stable facet. The two
bodies are indistinguishable here, and the seam across the upper surface is where the cells become extremely dense.
\textbf{(c), (d)} The marked regions at $\times6$ and $\times500$: the cells run
from $4\cdot10^{-2}$ across, where the equal-sag branch of \eqref{eq:celllaw}
binds, to $\approx10^{-4}$ along the ridge, where the tilt-adaptive branch does.}
\end{figure}

\section{Equilibria of convex polytopes}\label{sec:prelim}

Let $P\subset\R^3$ be a convex polytope with nonempty interior, $\com$ an interior
point (for us always the uniform-density centroid), and let
$N_P(x)=\{w : \langle w, y-x\rangle\le 0\ \forall y\in P\}$ denote the outward normal
cone at $x\in\bd P$. A point $x^\ast\in\bd P$ is an \emph{equilibrium} of $P$ with
respect to $\com$ if
\[
  x^\ast-\com \in N_P(x^\ast),
\]
i.e.\ if the plane through $x^\ast$ orthogonal to $x^\ast-\com$ supports $P$. Equilibria
are the critical points of $\delta(x)=|x-\com|$ on $\bd P$. On a polytope they are
located as follows; all three characterisations are classical and standard:

\begin{itemize}
\item \emph{Facets (stable).} For a facet $\mathcal F$ with outward unit normal $n$ and
support value $h=\max_{y\in P}\langle n,y\rangle$, the orthogonal projection (the
``foot'') of $\com$ onto the facet plane is
$q=\com+(h-\langle n,\com\rangle)\,n$. The facet carries an equilibrium iff
$q\in\mathcal F$; it is a nondegenerate local minimum of $\delta$ iff $q$ lies in the
relative interior of $\mathcal F$. Its \emph{margin} is the distance from $q$ to the
boundary of the polygon $\mathcal F$.
\item \emph{Edges (saddles).} For an edge $e=[v,w]$ shared by facets with normals
$n_A,n_B$, let $q$ be the foot of $\com$ on the line through $v,w$. Then the relative
interior of $e$ carries an equilibrium iff $q$ lies in the open segment and $q-\com$
lies in the closed planar cone spanned by $n_A,n_B$ (all three vectors are orthogonal
to $w-v$, exactly so in rational arithmetic); the equilibrium is nondegenerate, and
then automatically a saddle, iff $q-\com$ lies in the interior of that cone.
\item \emph{Vertices (unstable).} A vertex $v$ carries an equilibrium iff
$v-\com\in N_P(v)=\mathrm{cone}(n_{i_1},\dots,n_{i_k})$, the cone of its incident facet
normals; if $v-\com$ lies in the interior of this cone, $v$ is a nondegenerate local
maximum. For a \emph{simple} vertex ($k=3$) this is three sign conditions, and is
equivalent to the statement that $v$ maximises the linear functional
$\langle v-\com,\,\cdot\,\rangle$ over $P$; consequently it can also be tested
\emph{locally}: $v$ is a nondegenerate local maximum iff
$\langle v-\com,\, w-v\rangle<0$ for the three neighbouring vertices $w$ (with
$\le$ in place of $<$, the test characterises equilibria including degenerate ones).
\end{itemize}

The verifier of Section~\ref{sec:certification} decides the strict forms of these
tests and aborts on exact ties, so degenerate configurations, should they occur,
are detected rather than misclassified.

When all equilibria are nondegenerate, \eqref{eq:PH} holds with $S,H,U$ counting
facet-, edge- and vertex-equilibria. A polytope with $\shu=(1,0,1)$ is mono-monostatic:
$\delta$ has a single minimum and a single maximum on a sphere, and its gradient flow
carries every non-maximal point to the unique stable facet. \footnote{The discrete analogue, quasi-static \emph{rolling}, strictly lowers the centroid
at each step, so on $P$ every rolling sequence terminates on the unique stable
facet; we have not computed the resulting rolling graph, and we leave this for future work (for rolling dynamics
of polytopes see for instance \cite{rolling}).}

\section{The construction}\label{sec:construction}

\subsection{The smooth seed}
Our starting point is an explicit smooth mono-monostatic body $K$, found seeding with Sloan's analytic G\"omb\"oc-type parameterisation
$r^4=1+A\sin\varphi\,\cos(\theta-k\varphi)$ \cite{Sloan2023}, i.e.~the case
$A=0.12$, $k=5$ projected onto the harmonic basis below, and then searching
within a finite-dimensional family of star-shaped bodies given by
$r(u) = 1+\sum_{\ell=2}^{5}\sum_{m}c_{\ell m}Y_{\ell m}(u)$; this search over real spherical
harmonics involves $32$ coefficients and the absence of $\ell\in\{0,1\}$ keeps the centroid near
the origin. The coefficients were optimised so that the rigorous equilibrium count of
the smooth surface is $(1,0,1)$. The particular body used
throughout, which we denote $K_\ast$, has volume $\approx 4.22$ and is fixed by the
published coefficient vector.

\begin{remark}[$K_\ast$ is not convex]\label{rem:dimple}
The search that we performed enforced $r>0$ but not convexity, leading to a $K_\ast$
which fails to be convex: its minimal principal curvature is $\approx-1.9\cdot10^{-2}$
(median $+0.85$, maximum $+1.71$) on one shallow dimple about $5^\circ$ across, centred
$3.6^\circ$ from $q_S$. This is not to be worried about for the following reasons. First, the
dimple sits where the tilt is large, $\beta\approx5.4\cdot10^{-2}\approx900\,\beta_{\min}$, 
i.e.~nowhere near the near-critical ridge on which the construction and the equilibrium
structure depend. Second, it lies beneath the single designated stable facet, and the
cover's normal-space exclusion (Section~\ref{sec:construction}) removes every candidate
plane over the whole dimple, so no facet of $P$ is placed there. Third, 
$K_\ast$ leaves $P$ only across that facet's plane, on a patch of solid angle
$<3\cdot10^{-3}$~sr, by at most $3.3\cdot10^{-5}$ -- a set of volume $<4\cdot10^{-8}$,
and a protrusion $140$ times smaller than the certified foot margin of the same facet.
So $P$ is bounded by planes tangent to $K_\ast$ but does not quite contain it. Finally, and most importantly, no
certified claim refers to $K_\ast$ (Section~\ref{sec:certification}).
\end{remark}

For a smooth body $K$, star-shaped with respect to its centroid $\com$, with boundary the
radial graph $x(u)=r(u)\,u$, $u\in S^2$, define the \emph{tilt} $\beta(u)\in[0,\pi/2)$ at
$x(u)$ as the angle between the outer unit normal at $x(u)$ and the radial direction
$x(u)-\com$. Equilibria of $K$ are
exactly the zeros of $\beta$. The quantity governing everything below is
\[
  \beta_{\min} := \min\{\beta(u) : u \text{ outside small caps around the two equilibria}\},
\]
the minimal tilt on the ``ridge'' separating the basins of the two equilibria. For
$K_\ast$, $\beta_{\min}\approx 6.2\cdot10^{-5}$ rad. This is not an accident of our
seed: experimentally, every smooth $(1,0,1)$ body we produced in several parametric
families had $\beta_{\min}\lesssim10^{-3}$, and within our family, attempts to enlarge
$\beta_{\min}$ under the $(1,0,1)$ constraint failed systematically; mono-monostatic
bodies live close to the boundary of their own existence region, and the nearly-critical
ridge appears to be the geometric price (for the smooth G\"omb\"oc phenomenon of
``flatness'' see the discussion in \cite{VD2006}).

About its centroid, $P$ has $R_{\max}/R_{\min}=1.3298$ (the extrema
attained, as they must be for a $(1,0,1)$ body, at the two equilibria), and
its symmetry group is trivial -- consistent with the characterisation of the
symmetry groups of mono-monostatic bodies in \cite{DLV2021}, and inherited by
the polytopes below.

\subsection{Why naive faceting fails, quantitatively}
Let $P$ be a polytope bounded by planes tangent to $K$ (cf. Remark~\ref{rem:dimple}), with
tangency points spaced $h$ apart near a boundary point where the tilt is $\beta$ and the
relevant normal curvature is $\kappa$. Two elementary local computations explain all the
failures we observed. The fact that fine discretisation generically introduces additional equilibriais a known phenomenon, studied
quantitatively in \cite{DLSz2012} (see also \cite{DLSz2014}); the two
observations below are local versions of it, in the context of specific covers and sharp enough to be dealt with.

\begin{observation}[spurious vertex maxima]\label{obs:hcrit}
A vertex formed by tangent planes whose tangency points have spacing $h$ is a local
maximum of $\delta$ iff the radial direction lies in its normal cone; for a
near-critical point this happens iff
\[
  h \;\gtrsim\; h_{\mathrm{crit}} \;=\; \frac{2\beta}{\kappa}.
\]
Below $h_{\mathrm{crit}}$ the spurious maximum disappears.
\end{observation}

At the ridge minimum of $K_\ast$, $h_{\mathrm{crit}}\approx 10^{-4}$ rad: the cover must
be \emph{locally} that fine, but only there -- an adaptive cover with cell size
proportional to $\beta(u)$ concentrates the faces along the ridge. We emphasise a
finding that shaped the project: this obstruction is \emph{geometric, not numerical}. We
recomputed the decisive cone tests in $50$-digit arithmetic and in double precision; the
two agree to $10^{-16}$ at every cell size, with the sign flip occurring at
$h_{\mathrm{crit}}$, far above any rounding scale. No amount of arithmetic precision
substitutes for the missing faces.

\begin{observation}[sag dipole]\label{obs:dipole}
A tangent facet with tangency-point spacing $\rho$ ``sags'' outside $K$ by
$t\approx\kappa\rho^2/8$. If the cell size is chosen adaptively, $\rho\propto\beta$,
then $t\propto\beta^2$ inherits the asymmetry of the tilt field, and the union of sag
wedges displaces the centroid of $P$ from that of $K$ by a nonzero \emph{dipole}
moment. For $K_\ast$ with the pure adaptive rule the displacement is
$|\com_P-\com_K|\approx3.5\cdot10^{-4}$, while the ridge tolerates only displacements
$\ll R\,\beta_{\min}\approx6.6\cdot10^{-5}$: the equilibrium count near the ridge is
hypersensitive to the centre, and the naively covered polytope is not mono-monostatic
even when Observation~\ref{obs:hcrit} is respected.
\end{observation}

We verified Observation~\ref{obs:dipole} quantitatively: the measured displacement
agrees with the predicted sag-dipole integral in direction to within $12^\circ$
(cosine $0.979$).

\subsection{The adaptive equal-sag cover}
The final construction chooses, for each direction $u$, the local cell size
\begin{equation}\label{eq:celllaw}
  \rho(u) \;=\; \min\Bigl(\;\underbrace{\sqrt{8t_0/\kappa(u)}}_{\text{equal sag}},\;\;
  \underbrace{\max\bigl(\beta(u),\beta_0\bigr)/\sigma}_{\text{ridge-adaptive}}\Bigr),
\end{equation}
and places tangent planes at a maximal set of directions with pairwise angular
separations respecting~\eqref{eq:celllaw} (a graded Poisson-disc sampling with an
explicit \emph{coverage scrutiny}: deterministic probe rings around every tilt minimum
verify that no gap exceeds the local $\rho$, and any uncovered probe is added to the
cover). The first branch of \eqref{eq:celllaw} makes the sag \emph{constant} ($=t_0$)
away from the ridge, cancelling the dipole of Observation~\ref{obs:dipole} by symmetry;
the second branch enforces Observation~\ref{obs:hcrit} in the ridge band, with safety
factor $\sigma$ and a floor $\beta_0\propto\beta_{\min}$. Two special regions are
treated separately: the stable equilibrium of $K$ receives a single large tangent facet
(the designated stable face), and the unstable equilibrium is left as the natural apex
vertex of the surrounding cover; the exclusion is applied both in direction space and in normal space -- every cover plane
whose normal lies within $\mathrm{CAP\_RAD}_S$ of the cap normal is dropped -- so the cap
plane alone represents the whole neighbourhood of $q_S$ on the Gauss sphere.

The residual centroid displacement is then corrected by a \emph{compensation} step: the
off-ridge plane offsets are perturbed by a dipole field
$h_i\mapsto h_i+\zeta\,\langle n_i,d\rangle\,w(\beta_i)$ supported away from the
marginal band, with the scalar $\zeta$ solved by a secant iteration against the measured
centroid of the built polytope. This step is what allows aggressive parameters: with
$\sigma=1.5$, $t_0=10^{-4}$ the construction yields the $56{,}946$-face body of
Theorem~\ref{thm:main} with $|\com_P-\com_K|\approx1.5\cdot10^{-6}$, comfortably inside
the ridge tolerance; the conservative parameters $\sigma=3$, $t_0=10^{-5}$ and no
compensation give the $272{,}971$-face body. Within this construction family the face
count cannot be reduced much further: the ridge band alone requires
$\Theta(\sigma^2/\beta_{\min})$ faces, about $5\cdot10^4$ here, and pruning experiments
confirm a floor at $\approx5.5\cdot10^4$.

\begin{remark}
Every count above, produced during the search, was computed by a floating-point
classifier (a robust implementation of the criteria of Section~\ref{sec:prelim},
cross-validated on known polyhedra including Reshetov's). We treated its output as a
\emph{conjecture generator}: the construction was iterated until the classifier reported
$(1,0,1)$ robustly, and only then handed to the certifier. The same caveat applies to
every numerical diagnostic quoted in this section and in
Section~\ref{sec:frontier} ($\beta_{\min}$, the dipole magnitude and direction,
persistence values): these are floating-point measurements from the search phase,
recorded to motivate the construction and the open problems; they are not part of the
certified claims.
\end{remark}

\section{Exact certification}\label{sec:certification}

\subsection{The object and the claim}
A double-precision number is a dyadic rational, so the constructed data define
\emph{exactly}
\[
  P=\bigcap_{i=1}^{M}\{x : \langle n_i,x\rangle\le h_i\},\qquad n_i\in\Q^3,\ h_i\in\Q,
\]
with $M=56{,}946$. All claims of Theorem~\ref{thm:main} are about this rational
polytope; the smooth body $K_\ast$ plays no further role (it was scaffolding); 
in particular no property of $K_\ast$, convexity included, enters any certified claim. The
verification must establish: the combinatorial boundary structure of $P$; its exact
centroid; and the facet/edge/vertex equilibrium tests of Section~\ref{sec:prelim}, each
with a strict, quantified margin.

\subsection{Architecture: an untrusted oracle and exact decisions}
The verifier follows the \emph{certifying algorithm} paradigm \cite{MMNS2011}. A
floating-point convex-hull code (\textsc{Qhull} \cite{qhull}, via its polar-dual
formulation) is used only as an \emph{oracle} that proposes a combinatorial structure:
a list of triples $\{i,j,k\}$ of plane indices, each intended to define a vertex. Nothing
is trusted from the oracle; every subsequent decision is made either
\begin{itemize}
\item[(a)] in exact rational arithmetic (GMP), or
\item[(b)] in double precision under an explicit forward error bound (of Higham type
\cite{Higham}, with all input magnitudes checked at load time), used \emph{only in one
direction}: to bound the set of candidates that are then decided exactly. A
floating-point filter may never reject a potential equilibrium; it may only certify,
with a margin at least three orders of magnitude above the proven error bound (four,
for all filters except the saddle-numerator filter of phase~5), that a test is far
from tight.
\end{itemize}

\subsection{The boundary complex}
The oracle's triples are certified to be the true boundary complex of $P$ by elementary
convex geometry, avoiding any global convexity or immersion theorem:

\begin{lemma}\label{lem:complex}
Suppose a set $T$ of plane triples satisfies: (i) each triple's normals are linearly
independent (nonzero $3\times3$ determinant over $\Q$) and its intersection point
$v_t$ satisfies $\langle n_m,v_t\rangle\le h_m$ for \emph{all} $m$, strictly unless
$m\in t$; (ii) the plane-pair edges derived from $T$ form a closed connected
$2$-manifold with Euler characteristic $2$ in which every triple has exactly three
neighbours. Then every $v_t$ is a simple vertex of $P$, every derived edge is an edge of
$P$, and these are \emph{all} vertices and edges of $P$; each plane incident to a cycle
of at least three triples carries a facet, and $P$ has no other facets.
\end{lemma}

\begin{proof}[Proof sketch]
A feasible point with three linearly independent active constraints is a vertex of $P$,
and the strict inequalities in (i) make each $v_t$ simple and the $v_t$ pairwise
distinct (if two triples met at one point, a plane of one triple outside the other
would be active there, violating strictness). Next, each derived edge is a
\emph{bounded} edge of $P$: for a plane pair $\{i,j\}$ shared by triples $t\ne t'$, the
face $P\cap\{\langle n_i,x\rangle=h_i,\ \langle n_j,x\rangle=h_j\}$ is a convex subset
of a line containing the distinct vertices $v_t,v_{t'}$; vertices are extreme points,
hence endpoints of that face, which is therefore exactly the segment $[v_t,v_{t'}]$.
A simple vertex has exactly three edges, one per pair of active planes, so \emph{every}
edge of $P$ at a proposed vertex is a derived edge -- bounded, with both endpoints
proposed. Thus the proposed set is closed under adjacency in the graph of $P$, and no
proposed vertex meets an unbounded edge. Now $P$ is pointed (it has a vertex), and the
graph of vertices and \emph{bounded} edges of a pointed polyhedron is connected: a
generic linear functional negative on the recession cone is bounded above on $P$ and
increases, from any non-optimal vertex, along some edge, which cannot be an extreme ray
(the functional strictly decreases along recession directions); monotone edge paths
therefore join every vertex to the optimal one. Hence the proposed vertices exhaust
the vertices of $P$; and $P$ is bounded, since an unbounded pointed polyhedron carries
an extreme ray emanating from some vertex, while all vertices here meet only bounded
edges. (For the resulting polytope, connectivity is Balinski's theorem
\cite{Balinski}.) Facet identification follows because three distinct vertices of $P$
on a plane cannot be collinear, and the facet count is then pinned by Euler's relation.
\end{proof}

Condition (i) involves $M\times|T|\approx6.5\cdot10^9$ inequalities; the floating-point
filter certifies all but a few thousand of them, and the borderline pairs (in our run, $233$;
for the larger body, $2{,}622$) are decided exactly. Condition (ii) is a finite combinatorial check. In
addition the verifier certifies, in exact arithmetic, the strict local convexity of the
dual complex at each of the $170{,}832$ edges - not needed for
Lemma~\ref{lem:complex}, but a wholesale consistency check of the oracle output.

\subsection{The centroid}
The exact centroid is a ratio of polynomial expressions in the vertex coordinates
(themselves ratios of $3\times3$ determinants); computing it as a single rational
number invites coefficient explosion. Instead the verifier fan-triangulates each facet
(orientations verified exactly), accumulates the exact tetrahedral volume and moment
contributions in a fixed-point representation with denominator $2^{320}$, and obtains an
\emph{enclosure}: a rational point $\tilde\com$ with a proven bound
$|\com-\tilde\com|\le r$, $r\approx 2\cdot10^{-93}$. All equilibrium tests are then
performed exactly at $\tilde\com$, and every decisive quantity $q$ is additionally
required to satisfy $|q|>10^{-60}$. Each such $q$ is polynomial in the centre, with
gradient bounded on our data, from the per-coordinate magnitude bounds
checked at load time, by $20$ (foot slacks by $2.02$, vertex-cone margins by an
edge length $\le2\sqrt3\cdot1.25\le4.34$, edge-wedge determinants by $20$); since
$20r<10^{-90}\ll10^{-60}$, the sign of $q$ at $\tilde\com$ equals its sign at the
true centroid, and degenerate configurations would be detected, not misclassified. The verifier records the
smallest guarded magnitude of each run in the certificate, making the distance to
degeneracy a published, reproducible quantity: no decisive quantity came within
thirty orders of magnitude of the guard (minimum $|q|$ was $7.2\cdot10^{-20}$ for
$P_{56946}$ and $4.0\cdot10^{-21}$ for $P_{272971}$).

\subsection{The three counts}
With the complex and centroid certified, the tests of Section~\ref{sec:prelim} are run
for every facet, edge and vertex.
\emph{Stable:} for each of the $M$ planes, the foot of $\tilde\com$ is tested for
membership in $P$; the filter leaves a single candidate, the designated stable
facet, whose foot is then verified \emph{exactly against all other $M-1$ planes},
with minimal slack $4.65\cdot10^{-3}$; every other plane is excluded with certified
margin.
\emph{Unstable:} each vertex is tested by the three local inequalities
$\langle v-\tilde\com,\,w-v\rangle<0$; the filter leaves $106$ candidates, of which
exactly one survives exact adjudication (margins
$1.69\cdot10^{-8}$, $8.5\cdot10^{-7}$, $7.2\cdot10^{-6}$).
\emph{Saddle:} for each edge, the foot parameter condition is filtered on its
\emph{numerator} (note that the associated quotient is ill-conditioned for short edges); 
all $25{,}326$ surviving candidate edges are decided in exact
arithmetic, and none is a saddle. Finally $S-H+U=1-0+1=2$ is checked against
\eqref{eq:PH} -- by this point a consistency identity rather than an independent
test, since each count is already gated in its own phase -- and the entire chain is
bound to the input data by a cryptographic hash, so that checkpoints belonging to
different inputs cannot be mixed. The hash binding is a consistency device, not
tamper-proofing: in the certifying-algorithm trust model, the consumer of a
certificate does not trust shipped artefacts at all, but re-runs the verifier on the
published data (minutes on a laptop) and audits the verifier itself.

\subsection{What is trusted}
The certificate rests on: the $\approx700$-line verifier (published), the GMP rational
arithmetic library, IEEE-754 conformance of the hardware within the proven error
budgets, and the elementary facts quoted in Section~\ref{sec:prelim} and
Lemma~\ref{lem:complex}. It does \emph{not} rest on \textsc{Qhull}, on any
floating-point geometry, or on the software that produced the construction. A machine-checked
formalisation of the verifier's logic (e.g.\ in \textsc{Lean}) would be the natural next
rung on the assurance ladder and appears entirely feasible, since the mathematical
content is elementary.

\section{The human--AI workflow}\label{sec:workflow}

We now describe the genesis of our result in some detail because we believe the 
\emph{process} might be of independent interest.

\subsection{Division of labour}
The author posed the problem (construct an explicit mono-monostatic polyhedron;
then minimise faces; then certify), supplied the relevant literature, set the standards of
rigour, allocated computing resources, and redirected the strategy at some key
decision points. Everything else, i.e.~the smooth-body searches, the discretisation
experiments, the diagnosis of the obstructions of Section~\ref{sec:construction}, the
design of the equal-sag cover and of the compensation step, the floating-point
classifier and the certifier, was designed, implemented,
run and debugged by a general-purpose AI agent (Claude, Anthropic) over roughly $20$ working sessions, with the author reviewing summaries and directing at session
granularity. The agent also wrote the first draft of this paper.

Three episodes deserve record.

\subsubsection{Hypothesis testing}
When the naive discretisations failed, the working hypothesis ``the obstruction is
numerical precision'' was natural (and was the author's first guess). The agent's
decisive contribution was a numerical experiment: a local model of the ridge in
which the cone test of Observation~\ref{obs:hcrit} was evaluated in both double and
$50$-digit precision, demonstrating bit-identical margins with a sign change at a finite
cell size. This redirected the programme from arithmetic to geometry, i.e.~from ``more
digits'' to ``more faces, in the right places'', and ultimately to the cell
law~\eqref{eq:celllaw}.

\subsubsection{Verifying the verifier}
The first version of the certifier passed all its phases and produced the desired
certificate. Before accepting it, the agent subjected the verifier to an adversarial
review by independent instances of itself, prompted to \emph{refute} the certificate.
The review found three genuine soundness gaps: an ill-conditioned filter that could
have silently exempted tens of thousands of edges from exact adjudication, a reliance
on unverified adjacency data from the oracle, and a missing transfer bound from the
approximate to the exact centroid. All three were repaired (Section~\ref{sec:certification});
the repaired verifier reproduced the certificate, now soundly. We draw the obvious
moral: \emph{a certificate is only as good as its verifier, and verifiers written by
the same process that produced the object deserve independent hostile scrutiny.}

\subsubsection{Reward hacking}
During the (still unsuccessful) search for a mono-monostatic polyhedron with \emph{few}
faces, an optimiser reported a $19$-face solution within seconds. It was an artefact in
the most literal sense: an unbounded, combinatorially broken region (Euler
characteristic $-15$) on which the floating-point classifier's counts were meaningless;
the optimiser had found a hole in the sanity checks faster than it could find a
polyhedron. Numerical search over geometric objects offers many such traps, and we found
the only reliable antidote to be the one advocated throughout: hard invariant gates
during search, and exact certification before belief.

\subsection{Cost}
The entire programme consumed some days (maximum a couple of weeks) of a single laptop and of some batch
allocations on a compute cluster; the certification itself runs in minutes.

\section{Towards a better G\"omb\"ocedron}\label{sec:frontier}

Our construction is certainly far from optimal. We record, as honestly as we can, what
the search phase of this project established about the gap between $5\cdot10^4$ faces
and the conjectured optimum in the tens of faces.

\emph{The smooth-approximation route has a floor.} Within tangent covers of smooth
mono-monostatic bodies, the ridge band forces $\Theta(\sigma^2/\beta_{\min})$ faces, and
$\beta_{\min}\lesssim10^{-3}$ appears unavoidable for smooth $(1,0,1)$ bodies. Unless a
smooth seed with a radically larger ridge tilt exists (our searches say no), this route
cannot go far below $\sim5\cdot10^4$ faces. Few-face mono-monostatic polyhedra, if
found, will not resemble discretised G\"omb\"ocs.

\emph{Reshetov's body is one Morse cancellation away; but a distant one.} We certified,
running the verifier of Section~\ref{sec:certification} with expected signature
$(1,1,2)$ on a dyadic-rational realisation of the published integer data (the
floating-point normalisation and recentring perturb the data at relative scale
$10^{-16}$, while every sign decision of the certification exceeds $2.1\cdot10^{-7}$
in magnitude, with the smallest being recorded in the certificate, so the signature is
that of Reshetov's body), that the $14$-face unistable polyhedron of
\cite{Reshetov2014} has signature $(1,1,2)$: one surplus saddle--maximum pair. In Morse-theoretic terms the pair could be cancelled;
quantitatively, its \emph{persistence} (the gap in $\delta$ between the saddle and its
partner maximum) is enormous - roughly $24$ in units where the body has unit volume,
because the body is a $17:1$ elongated ``canoe'' whose surplus maximum is a far tip.
Cancellation would require un-elongating the body entirely, against the $13$
foot-ejection constraints that hold its unistability, each with relative margin only
$\sim10^{-3}$. Our attempts (structured perturbations, added degrees of freedom,
persistence-driven objectives) consistently confirmed this rigidity.

\emph{A chunky attractor at $F\approx17$.} Direct parametric searches over structured
half-space families (spiral and ``ring-cam'' polytopes with a designated stable facet
and apex vertex, guided by a graded objective that penalises surplus equilibria by their
geometric extinction depths and surplus pairs by their persistence) reliably produce
compact bodies with signatures $(2,3,3)$ and $(3,3,2)$ at around $17$ faces, with some
surplus pairs of very small persistence ($<10^{-2}$) - but the final cancellations have
so far resisted several hundred independent optimisation runs. Whether the obstruction
is again structural or merely a failure of our search remains open; we record the
question:

\begin{quote}
\emph{Does a (homogeneous) mono-monostatic polyhedron with at most $20$ faces exist? More modestly:
does one with fewer than $10^3$ faces exist?}
\end{quote}

Any candidate, from any method, can be certified by the verifier of
Section~\ref{sec:certification} in seconds at these sizes; we hope the availability of a certificate lowers the barrier for attempts on the record.

\section{Data and reproducibility}\label{sec:data}

The following artefacts accompany the paper: the half-space data of both certified
polytopes (with SHA-256 digests), the certificates (machine-readable JSON and
human-readable form), the exact $(1,1,2)$ certificate of the dyadic realisation of
Reshetov's body together with its half-space data, the verifier source, the
construction pipeline, and the smooth seed coefficients. The certification chain is deterministic; running it regenerates the
certificate from the half-space data alone. The construction pipeline is likewise
deterministic given its published parameters. Everything is available at \href{\repourl}{\nolinkurl{\repourl}}.

\appendix

\section{Algorithmic details}\label{app:algo}

\subsection{Cover generation}
Algorithm~\ref{alg:cover} summarises the construction of
Section~\ref{sec:construction}. The tilt field $\beta$, curvature field $\kappa$ and all
tangency data refer to the smooth seed $K_\ast$; ``maximal graded sample'' means a
greedy Poisson-disc thinning, finest cells first, of an adaptively refined candidate
pool whose local density everywhere exceeds the target~\eqref{eq:celllaw}. The
refinement of the candidate pool must carry each local minimum of $\beta$ in its
frontier (an early version dropped the minimiser itself from the refinement queue,
leaving an undetected hole of twenty times the critical size at the bottom of the
ridge - found only by the coverage audit of step 4, which we therefore regard as
non-optional).

\begin{algorithm}
\caption{Equal-sag adaptive tangent cover}\label{alg:cover}
\begin{algorithmic}[1]
\State locate the equilibria and the ridge of $K_\ast$; measure $\beta_{\min}$
\State build a graded candidate pool of directions, refined until the local spacing is
below $\rho(u)$ of \eqref{eq:celllaw} everywhere
\State greedily select a maximal subset with pairwise separations $\ge\rho$; place a
tangent plane at each selected direction; add the stable cap plane
\State \textbf{audit}: probe rings around every tilt minimum; verify no coverage gap
exceeds the local $\rho$; append tangent planes at any uncovered probe
\State measure the centroid of the built polytope; apply the dipole compensation
$h_i \mathrel{+}= \zeta\langle n_i,d\rangle w(\beta_i)$, resolving $\zeta$ by secant
iteration (two to three builds)
\State classify with the floating-point classifier; accept iff $(1,0,1)$ with robust
margins
\end{algorithmic}
\end{algorithm}

\subsection{The floating-point classifier}
The search-time classifier implements the tests of Section~\ref{sec:prelim} on the
boundary complex obtained from the polar-dual convex hull (the dual points $n_i/h_i$ lie
near a sphere, so the hull is well-conditioned even when the primal tangent planes are
nearly parallel). Facet feet are tested against all half-spaces; vertex cones by
non-negative least squares; edge wedges by planar cross-ratios. It was validated by
reproducing the known signatures of a suite of small polyhedra, including Reshetov's
$(1,1,2)$ -- a signature we subsequently certified in exact arithmetic
(Section~\ref{sec:frontier}), so this validation case no longer rests on the
classifier itself.

\subsection{The certifier}
The verifier runs in seven phases, each writing a checkpoint bound to the SHA-256 of
the input arrays: (0) oracle hull; (1) boundary complex certification
(Lemma~\ref{lem:complex}: exact orientation and independence determinants, derived edge
combinatorics, bulk feasibility filter at threshold $10^{-10}$ against a proven error
bound $<10^{-14}$, exact adjudication of all borderline pairs); (2) centroid enclosure
(exact tetrahedral moments, fixed-point accumulation at $2^{-320}$, enclosure radius
$\approx2\cdot10^{-93}$); (3) stable count (filter, then the single candidate verified
exactly against all planes); (4) unstable count (local three-inequality tests; all
filter survivors adjudicated exactly); (5) saddle count (numerator-filtered foot test,
sound for all edge lengths; \emph{every} survivor adjudicated exactly); (6) report,
which assembles the certificate, rechecks $S-H+U=2$ as a consistency gate (each count
is already gated in its own phase), and records the transfer guards $|q|>10^{-60}$ on
every exact decision together with the smallest guarded magnitude of the run. The
expected signature is a parameter of the verifier (default $(1,0,1)$), so the same
pipeline certifies or refutes any claimed signature; the $(1,1,2)$ certificate of
Reshetov's body (Section~\ref{sec:frontier}) is produced by the identical chain, in
under a second. Total runtime: under ten minutes for the larger body, under four for
the smaller, on a laptop. Table~\ref{tab:bodies} collects the certified quantities of
both bodies.

\begin{table}[h]
\centering
\begin{tabular}{lcc}
\toprule
 & $P_{272971}$ & $P_{56946}$ \\
\midrule
faces / edges / vertices & $272{,}971$ / $818{,}907$ / $545{,}938$ & $56{,}946$ / $170{,}832$ / $113{,}888$ \\
volume & $4.217893406\ldots$ & $4.218597429\ldots$ \\
stable-facet margin & $4.319\cdot10^{-3}$ & $4.650\cdot10^{-3}$ \\
unstable-vertex margins (min) & $1.748\cdot10^{-8}$ & $1.695\cdot10^{-8}$ \\
saddle candidates decided exactly & $52{,}958$ & $25{,}326$ \\
centroid enclosure radius & $\le1.01\cdot10^{-92}$ & $\le2.11\cdot10^{-93}$ \\
borderline feasibility pairs & $2{,}622$ & $233$ \\
\bottomrule
\end{tabular}
\medskip
\caption{The two certified mono-monostatic polytopes. Margins are the certified
minimal slacks, rounded down: stable $=\min_j\,(h_j-\langle n_j,q\rangle)$ over all
competing planes $j$ (the normals are unit to within $1\%$, so Euclidean point--plane
distances agree with these values to that factor), unstable $=$ the normal-cone
slacks $-\langle v-\tilde\com,\,w-v\rangle$ of Section~\ref{sec:prelim}.}
\label{tab:bodies}
\end{table}

\end{document}